\documentclass[12pt,letterpaper]{amsart}

\usepackage{euscript,amsfonts,amssymb,amsmath,amscd,amsthm,enumerate,hyperref}

\usepackage{comment}

\usepackage{graphicx}

\usepackage{color}

\usepackage[margin=1.1in]{geometry}

\usepackage{bbm}

\newtheorem{theorem}{Theorem}[section]
\newtheorem*{theorem*}{Theorem}
\newtheorem{lemma}[theorem]{Lemma}
\newtheorem{definition}[theorem]{Definition}
\newtheorem{proposition}[theorem]{Proposition}
\newtheorem{conjecture}[theorem]{Conjecture}
\newtheorem{corollary}[theorem]{Corollary}

\theoremstyle{remark}
\newtheorem{remark}[theorem]{Remark}

\newcommand{\eps}{\varepsilon}

\newcommand{\E}{\mathbb E}

\newcommand{\eq}{\begin{equation}}
\newcommand{\en}{\end{equation}}

\renewcommand{\a}{\mathbf a}
\renewcommand{\b}{\mathbf b}

\newcommand{\G}{\mathcal G}

\newcommand{\HO}{\mathcal F}
\newcommand{\B}{\mathcal B}

\title{Crystallization of random matrix orbits}

\author{Vadim Gorin}

\address[Vadim Gorin]{Department of Mathematics, Massachusetts Institute of Technology, Cambridge, MA, USA, and
 Institute for Information Transmission Problems of Russian Academy of Sciences, Moscow, Russia. E-mail: vadicgor@gmail.com}

\author{Adam W. Marcus}

\address[Adam W. Marcus]{Department of Mathematics, Princeton University, Princeton,
NJ, USA}

\begin{document}

\begin{abstract}
 Three operations on eigenvalues of real/complex/quaternion (corresponding to $\beta=1,2,4$) matrices,
 obtained from cutting out principal corners, adding, and multiplying matrices can be extrapolated to general values of $\beta>0$ through associated special
 functions.
 We show that the $\beta\to\infty$ limit for these operations leads to the finite free projection, additive convolution,
 and multiplicative convolution, respectively.

 The limit is the most transparent for cutting out the corners, where the
 joint distribution of the eigenvalues of principal corners of a uniformly-random general $\beta$  self--adjoint  matrix with fixed eigenvalues is known as the $\beta$--corners process.
 We show that as
 $\beta\to\infty$ these eigenvalues crystallize on an irregular lattice consisting of the roots of derivatives of a single polynomial.
 In the second order, we observe a version of the
 discrete Gaussian Free Field (dGFF) put on top of this lattice, which provides a new explanation as to why the (continuous) Gaussian Free Field governs the global asymptotics of random matrix ensembles.

\end{abstract}

\maketitle

\section{Matrix operations at general $\beta$}
\label{Section_intro}

Fix $N$ and consider two $N\times N$ self--adjoint matrices $A_N$, $B_N$ with
either real, or complex, or quaternion entries. In this article we are mostly
interested in the (real) \emph{eigenvalues} of these matrices.
In particular, we consider three natural
matrix operations, which have non-trivial influence on the eigenvalues.
\begin{enumerate}
\item We can cut out the principal top--left $k\times k$ corner of the matrix,
$A_N\mapsto A_k$, where $A_k$ is the projection of $A_N$.
\item We can add the matrices, $(A_N,B_N)\mapsto A_N+B_N$,
\item We can multiply the matrices $(A_N,B_N)\mapsto A_N B_N$.
\end{enumerate}

In principle, all three operations can be expreseed in terms of multiplication, since for small
$\eps$, $(1+\eps A_N)(1+\eps B_N)\approx 1+\eps(A_N+B_N)$, thus reducing addition to
multiplication (all three operations can be expressed in terms of addition as well, but in a less obvious way, see (\ref{ex:petter})).
On the other hand, if the eigenvalues of $B_N$ are positive, then
$A_N B_N$ and $(B_N)^{1/2} A_N (B_N)^{1/2}$ differ by conjugation, and therefore
have the same spectrum.
If now $B_N$ is the matrix of the projector onto the first
$k$ basis vectors, then $(B_N)^{1/2} A_N (B_N)^{1/2}$ is precisely the $k\times k$
corner of $A_N$, i.e.\ the projection $A_k$. Nevertheless, we will consider all
three operations, as this will provide more insights.

For deterministic matrices, the relations between the spectra of $A_N$ and $A_k$ is
folklore; it is given by simple interlacing conditions (cf.\ \cite{Ner} and
Definition \ref{def_betacorner}). The result of the operation $(A_N,B_N)\to A_N+B_N$
on the spectrum is the subject of the celebrated \emph{Horn's (ex-)conjecture}, and
similar results are now known for $(A_N,B_N)\to A_N B_N$, see \cite{Fulton} for a
review.

Our point of view is different, as we consider \emph{random} $A_N$, $B_N$ with
\emph{invariant} distributions, which means that given the eigenvalues of a matrix,
its eigenvectors are conditionally uniform. This is the same as declaring the
distribution to be invariant under the action of orthogonal/unitary/symplectic group
(depending on the base field) by conjugations --- hence, the name.

It suffices to study the case when the eigenvalues of $A_N$ and $B_N$ are
\emph{deterministic}, since other cases can be obtained as mixtures. We therefore
fix two $N$--tuples of reals $\a=(a_1,\dots,a_N)$, $\b =(b_1,\dots,b_N)$ and define
$A_N$, $B_N$, to be \emph{uniformly random independent} matrices with corresponding
prescribed eigenvalues.

For each of the three operations on matrices, we arrive at an operation of the
eigenvalues, whose result is a \emph{random spectrum}.

\begin{enumerate}
\item $\a \mapsto \pi^\beta_{N\to k}(\a)=(a_1^{(k)},\dots,a_k^{(k)})$,
where the latter is the random $k$ (real) eigenvalues of $A_k$, the corner of $A_N$;

\item $(\a,\b)\mapsto \a \boxplus_{\beta} \b$, where the
latter is the random $N$ (real) eigenvalues of $A_N+B_N$

\item Assume that all eigenvalues in $\a$, $\b$ are \emph{positive},
and define $(\a,\b)\mapsto \a \boxtimes_\beta \b$, where the latter is the random
$N$ (real) eigenvalues of either $A_N B_N$, or $B_N A_N$, or $(B_N)^{1/2} A_N
(B_N)^{1/2}$, or $(A_N)^{1/2} B_N (A_N)^{1/2}$.  (the eigenvalues of all 4 matrices
are the same, and the last two matrices are self--adjoint, which shows that these
eigenvalues are real).
\end{enumerate}

The subscript $\beta$ in the operations $\pi^\beta_{N\to k}$, $\boxplus_{\beta}$,
$\boxtimes_{\beta}$ serves as an indication that they \emph{depend} on whether we
deal with real/complex/quaternion matrices, corresponding to $\beta=1,2,4$. More
generally, these operations can be \emph{extrapolated} to general values of the real
parameter $\beta>0$. For the projection $\pi^\beta_{N\to k}$ the result is known as
$\beta$--corners process, cf.\ \cite{Ner},\cite{GS},\cite{BG_GFF}. For addition
$\boxplus_{\beta}$ and multiplication $\boxtimes_{\beta}$ this is done by
identifying the random eigenvalues with their Laplace-type integral transforms
related to multivariate Bessel functions and Heckman--Opdam hypergeometric
functions, respectively. The operations then turn into simple multiplication of
these special functions and re-expansion of the result in terms of the same
functions\footnote{There is a tricky point in the definition: the positivity of the
coefficients in the re-expansion is a well-known conjecture, which is still open. In
the event that the positivity is not true for some values of $\beta$, the
distributions of $\mathbf a \boxplus_{\beta} \b$, $\a \boxtimes_{\beta} \b$ might fail
to be probability measures, but rather be tempered distributions, i.e.\ continuous
linear functionals on smooth test functions. For
 $\pi^\beta_{N\to k}$ the situation is simpler, as the available explicit formulas make positivity immediate.}, see Section \ref{Section_operations} for the details.

An alternative, yet conjectural, approach to the operations on random matrices at
general values of $\beta$ has been proposed in \cite{Edelman} where the framework of
$\beta$--\emph{ghosts and shadows} was developed. The idea there is to treat arbitrary
$\beta>0$ as the dimension of a (typically non-existent) real--division algebra, by expressing all
probabilistic properties of interest through Dirichlet distributions.
We do not
know whether the technique of \cite{Edelman} can be pushed through to the point of
reproducing the operations $\boxplus_\beta$, $\boxtimes_\beta$. 

Our first result concerns the dependence on $\beta$ of such operations.

\begin{theorem} \label{Theorem_independence}
 Let $z$ be a formal variable. For fixed $\a$ and $\b$,
 define the polynomials $Q^{N\to k}(z)$, $Q^{\boxplus}(z)$, $Q^{\boxtimes}(z)$ of
 degrees $k$, $N$, $N$, respectively, as expected characteristic polynomials of the
 corresponding matrices, i.e.\
 \begin{eqnarray}
 Q^{N\to k}(z)&=&\E \prod_{\alpha\in \pi^\beta_{N\to k}(\a)}
 (z-\alpha),\\
 Q^{\boxplus}(z)&=&\E \prod\limits_{\alpha\in \a \boxplus_{\beta} \b}
  (z-\alpha)\\
Q^{\boxtimes}(z)&=&\E \prod\limits_{\alpha\in \a \boxtimes_{\beta} \b}
  (z-\alpha).
 \end{eqnarray}
 Then the polynomials $Q^{N\to k}(z)$, $Q^{\boxplus}(z)$, $Q^{\boxtimes}(z)$ (i.e.\
 their coefficients) do not depend on the choice of $\beta>0$. They can be computed
 as follows:
 \begin{eqnarray}
 \label{eq_ff_projection} Q^{N\to k}(z)&=&\frac{1}{N(N-1)\cdots (k+1)} \left(\frac{\partial}{\partial
  z}\right)^{N-k} \prod_{i=1}^N (z-a_i)
 \\
 \label{eq_ff_add} Q^{\boxplus}(z)&=&\frac{1}{N!} \sum_{\sigma\in S_N} \prod_{i=1}^N
  (z-a_i-b_{\sigma(i)})\\
 \label{eq_ff_multiply}Q^{\boxtimes}(z)&=&\frac{1}{N!} \sum_{\sigma\in S_N} \prod_{i=1}^N
  (z-a_ib_{\sigma(i)})
 \end{eqnarray}
\end{theorem}
The proof of Theorem \ref{Theorem_independence} is given in Section
\ref{Section_Proof_of_Th_independence}.
Note that Theorem \ref{Theorem_independence} includes the fact that the expectations of the elementary symmetric functions in
the variables $\pi^\beta_{N\to k}(\a)$, $\a\boxplus_\beta \b$, $\a\boxtimes_\beta
\b$ do not depend on $\beta$. A similar observation was recently used in
\cite{Bor_independence} in the context of the Macdonald measures (following an
earlier observation \cite[Page 318]{M} that the application of the Macdonald
$q$--difference operator to the Macdonald reproducing kernel produces an independent
in $q$ factor). We believe that this is more than a coincidence, but rather a manifestation
of a general phenomenon. Indeed, Macdonald polynomials can be degenerated to
Heckman-Opdam and multivariate Bessel functions that we rely on, cf.\ \cite{BG_GFF},
\cite{YiSun_HO} and discussion in the next section. A parallel degeneration leads to
a class of Gibbs probability measures on (continuous) interlacing particle
configurations, as studied in \cite{OV}. On the other hand, the same article (mostly
for $\beta=2$, see however the very last paragraph there) explains that such
measures can be also obtained from $\beta$--corners processes with fixed top rows.
An extension of Theorem \ref{Theorem_independence} to Macdonald polynomials is also
explained in Section \ref{Section_Discrete}.

For classical Gaussian/Laguerre/Jacobi ensembles of random matrices the independence
on $\beta$ of the expectation of the characteristic polynomial was also previously
noticed by some authors, see e.g.\ \cite[Lemma 5.3]{BG_GFF}.
Theorem~\ref{Theorem_independence} in fact holds for a much larger class of
distributions. In \cite{Marcus2}, it is shown that
Theorem~\ref{Theorem_independence} holds for any distribution that is invariant
under conjugation by signed permutation matrices, and it has been noted in
\cite{PudSaw} that the same holds for any distribution that is invariant under
conjugation by matrices in the standard representation of $S_{N+1}$.

\smallskip

 The polynomial operations defined by (\ref{eq_ff_add}) and (\ref{eq_ff_multiply}) have a long history in the literature (dating at least back to \cite[Page 176]{Walsh}) due to an interest in understanding operations that preserve real rootedness of polynomials.
These ideas were more recently extended to the realm of {\em stable} polynomials (a
multivariate version of real rootedness) in a series of works by
 Borcea  and Br\"{a}nd\'{e}n, see, for example, \cite{BB}.
 One of their results implies that any convolution that treats the coefficients of polynomials linearly can be written using the additive convolution (sometimes at the cost of needing to use extra variables).
 In particular, for fixed $\a$ and $\b$, one has
 \begin{equation}\label{ex:petter}
 \sum_{\sigma\in S_N} \prod_{i=1}^N
  (z-a_ib_{\sigma(i)})
  =
  \left(\prod_{i=1}^N b_i \right) \sum_{\sigma\in S_N} \prod_{i=1}^N
  (y - a_i - z/b_{\sigma(i)}) \bigg|_{y = 0}
 \end{equation}
which, in theory, would allow one to compute $Q^{\boxtimes}(z)$ using $Q^{\boxplus}(z)$.
We have not found any advantage to treating both as $Q^{\boxplus}(z)$ (or treating both as $Q^{\boxtimes}(z)$ as was mentioned earlier).

 Recent interest in such operations has come from new techniques involving the expected characteristic polynomials of certain combinations of random matrices.
 In particular, Theorem \ref{Theorem_independence} links us directly to the finite free probability developed in
\cite{MSS}, \cite{Marcus}, where the operations producing $Q^{\boxplus}(z)$,
$Q^{\boxtimes}(z)$ from $\a$, $\b$ are called finite free additive and
multiplicative convolutions. These convolutions are a useful tool in the ``method of
interlacing polynomials'' as introduced in \cite{MSS1} --- in particular, the
additive convolution plays an important role in the same authors' proof of the
existence of Ramanujan (multi)graphs of any degree and any size \cite{MSS4}.

\bigskip

The second main result deals with $\beta\to\infty$ limit, a proof of which appears in Section
\ref{Section_Proof_of_Th_crystal_1}.

\begin{theorem} \label{Theorem_crystal} Fix $k\le N$, $\a$ and $\b$. The distributions
of $\pi^{\beta}_{N\to k}(\a)$, $\a \boxplus_\beta b$, $\a \boxtimes_\beta b$ (in
$k$, $N$, $N$ dimensional spaces, respectively) weakly converge on polynomial
test--functions as $\beta\to\infty$ to $\delta$--measures on the roots of
polynomials $Q^{N\to k}(z)$, $Q^{\boxplus}(z)$, $Q^{\boxtimes}(z)$ of Theorem
\ref{Theorem_independence}, respectively.
\end{theorem}
Theorem \ref{Theorem_crystal} implies that $Q^{N\to k}(z)$, $Q^{\boxplus}(z)$,
$Q^{\boxtimes}(z)$ are real-rooted polynomials. For $Q^{N\to k}(z)$ this is easy to
prove directly; proofs of the other two appear in \cite{MSS} using the techniques of \cite{BB}, but also follow directly from the (vastly more general) main result in \cite{MCPC}.

There is a remarkable link of Theorem \ref{Theorem_crystal} for $\pi^{\beta}_{N\to
k}(\a)$ to classical ensembles of random matrices. It is well-known (cf.\
\cite{Kerov}, \cite{DE-CLT}, \cite{BG_GFF}) that the eigenvalues in
Hermite/Laguerre/Jacobi $N$--particle ensembles (i.e.\
G$\beta$E/L$\beta$E/J$\beta$E) concentrate as $\beta\to\infty$ near the roots of the
corresponding orthogonal polynomials. Theorem \ref{Theorem_crystal} predicts that
varying $N$ should be the same as taking derivatives (see \cite{Ner} and
\cite{YiSun} for an explanation why classical ensembles agree with the operation
$\pi^\beta_{N\to k}$), and indeed, the derivatives of Hermite/Laguerre/Jacobi
orthogonal polynomials are again such orthogonal polynomials of smaller degree ---
this is a relation known as the forward shift operator, see e.g.\ \cite{KoSw}.

\bigskip

For the $\beta$--corners processes, capturing $\pi^{N\to k}_\beta(\a)$
simultaneously for all $k=1,\dots,N$ gives a more precise $\beta\to\infty$
asymptotic theorem, which we now present.
For each $N=1,2,\dots$, let $\G_N$ be the set of all Gelfand--Tsetlin patterns of
rank $N$, which are arrays $\{x_{i}^k\}_{1\le i \le k \le N}$ satisfying
$x^{k+1}_i\le x^{k}_i \le x^{k+1}_{i+1}$. We refer to the coordinate $x_i^k$ as the
position of the $i$th particle in the $k$th row.

\begin{definition}\label{def_betacorner} The $\beta$--corners process with top row
$y_1<\dots<y_N$  is the unique probability distribution on the arrays
$\{x^k_i\}_{1\leq i\leq k\leq N}\in \G_N$, such that $x^N_i=y_i$, $i=1,\dots,N$, and
the remaining $N(N-1)/2$ particles have the density
\begin{equation}
\label{eq_beta_corners_def}
 \frac{1}{Z_N}
  \prod_{k=1}^{N-1} \left[\prod_{1\le i<j\le k} (x_j^k-x_i^k)^{2-\beta}\right] \cdot \left[\prod_{a=1}^k \prod_{b=1}^{k+1}
 |x^k_a-x^{k+1}_b|^{\beta/2-1}\right],
\end{equation}
where $Z_N$ is the normalizing constant computed as
\begin{equation}
\label{eq_normalization}
 Z_N=\prod_{k=1}^N \frac{ \Gamma(\beta/2)^k}{\Gamma(k\beta/2)} \cdot \prod_{1\le i < j \le N}
 (y_j-y_i)^{\beta-1}.
\end{equation}
\end{definition}
\begin{remark}
 The inductive (in $N$) computation of $Z_N$ in \eqref{eq_normalization} is the
 Dixon--Anderson integration formula, see \cite{Dixon},\cite{Anderson},
\cite[chapter 4]{For}.
\end{remark}

When $\beta=1,2,4$, the $\beta$--corners process admits a random matrix realization.
In this setting, one considers a uniformly random self--adjoint real/complex/quaternion
(corresponding to $\beta=1,2,4,$ respectively) $N\times N$ matrix $[A_{i,j}]$ with
fixed eigenvalues $a_1,\dots,a_N$. Neretin shows in \cite{Ner}\footnote{ Neretin
remarks that the statement is a folklore going back at least to the work of Gelfand
and Naimark in 1950s.} (see also \cite{Baryshnikov} for $\beta=2$ case) that $x_i^k$
can be identified with $k$ eigenvalues of the $k\times k$ principal corner
$[A_{i,j}]_{i,j=1}^k$.

\begin{definition} \label{def_beta_infty_corner}
 The $\infty$--corners process with top row $\a=(a_1<\dots<a_N)$, is a deterministic
 array of particles $\{x_i^k\}_{1\le i \le k \le N}\in \G_N$ such that for each
 $k=1,2\dots,N$, ${x_1^k<x_2^k<\dots<x_k^k}$ are $k$ roots of $Q^{N\to k}(z)$, which means
 \begin{equation}
   \prod_{i=1}^k (z-x_i^k)= \frac{1}{N(N-1)\cdots (k+1)}
   \left(\frac{\partial}{\partial z}\right)^{N-k} \left[ \prod_{i=1}^N (z-a_i)
   \right],
 \end{equation}
and equipped with a Gaussian field $\{\xi_{i}^k\}_{1\le i \le k \le N}$, such that
$\xi_1^N=\xi_2^N=\dots=\xi^N_N=0$ and the remaining coordinates have density
proportional to
\begin{equation}
\label{eq_Gaussian_density}
 \exp\left(\sum_{k=1}^{N-1} \left[ \sum_{1\le i <j \le k} \frac{ (\xi_i^k-\xi_j^k)^2}{2 (x_i^k-x_j^k)^2}
 - \sum_{a=1}^k \sum_{b=1}^{k+1} \frac{ (\xi_a^k-\xi_b^{k+1})^2 }{4 (x_a^k-x_b^{k+1})^2}
 \right]\right).
\end{equation}
\end{definition}

 We call $\{x_i^k\}_{1\le i \le k \le N}\in \G_N$ the \emph{deterministic part} of the
 $\infty$--corners process and $\{\xi_{i}^k\}_{1\le i \le k \le N}$ the
 \emph{discrete Gaussian Free Field} on top of it.
 The fact that the $\infty$--corners process captures $\beta\to\infty$ behavior of the $\beta$--corners processes is proved in Section \ref{Section_Proof_of_Th_crystal_2}.

\begin{theorem} \label{Theorem_crystallization}
 Fix $y_1<\dots<y_N$, take a $\beta$--corners process $\{x_i^k(\beta)\}_{1\le i \le k \le N}$ with top row
 $y_1<\dots<y_N$, and $\infty$--corners process $\{\tilde x_i^k, \xi_i^k\}_{1\le i \le k \le
 N}$ with the same top row. Then we have the
 convergence in distribution:
 \begin{equation}
 \label{eq_crystallization_formula}
  \tilde x_i^k=\lim_{\beta\to\infty} x_i^k(\beta),\quad \quad \xi_i^k
  =\lim_{\beta\to\infty} \sqrt{\beta} \bigl(x_i^k(\beta)-\tilde x_i^k\bigr), \qquad 1\le i \le k \le
  N.
 \end{equation}
\end{theorem}
The proof of Theorem \ref{Theorem_crystallization} is given in Section
\ref{Section_proofs}. The difference between Theorem \ref{Theorem_crystal} for
$\pi^{N\to k}_\beta(\a)$ and Theorem \ref{Theorem_crystallization} is that the
latter captures not only a deterministic limit (Law of Large Numbers), but also
Gaussian fluctuations around it. It would be interesting to do the same for $\a
\boxplus_\beta \b$ and $\a \boxtimes_\beta \b$ as $\beta\to\infty$, but we do not
have any theorems in this direction so far.

\bigskip

There is a significant amount of literature studying the fluctuations of
$\beta$--corners processes as $N\to\infty$ with $\beta$ being fixed.
 When the top
row $\a$ is \emph{random} with specific distribution (rather than deterministically
fixed in our setting), then the asymptotic centered fluctuations were identified
with a pullback of the 2d Gaussian Free Field with Dirichlet boundary conditions in
\cite{Bor_CLT}, \cite{BG_GFF}, \cite{GZ}. For a discrete analogue of $\beta=2$ case
(cf.\ \cite[Section 1.3]{BG_LLN} for the link between discrete and continuous
setting) similar results are also known \cite{Petrov_GFF}, \cite{BG_CLT} in the
deterministic $\a$ case.
Despite these works, the conceptual reasons for the appearance of the
Gaussian Free Field in $\beta$--random matrices has remained somewhat unclear.

On the other hand, the $\beta=\infty$ case, which is the joint distribution of
$\{\xi^k_i\}$ in Definition \ref{def_betacorner} is a version of the \emph{Discrete}
Gaussian Free Field. In many examples the convergence of dGFF to GFF is known, cf.\
\cite{Sheffield}, \cite{Sheffield_Schramm}, \cite{Werner_GFF}, which provides a new
insight on why the continuous Gaussian Free Field should show up in $N\to\infty$
limit of the random matrices. Let us however emphasize that the convergence of the
$\infty$-corners process toward GFF does not follow from any known results, since
the points $\{x^k_i\}$, where dGFF lives, vary in a very non-trivial way as
$N\to\infty$, and, in addition, the exponent in \eqref{eq_Gaussian_density} has both
positive and negative terms.

\medskip

At $\beta=1,2,4$ the $N\to\infty$ limit of the operations $\a\boxplus_{\beta} \b$
and $\a\boxtimes_{\beta} \b$ are also well-studied. In particular, the first--order
behavior (``Law of Large Numbers'') is linked to the free probability theory cf.\
\cite{VDN}, \cite{NS}, and the limiting operations  are known as free (additive and
multiplicative) convolutions, see e.g.\ \cite{Vo}, \cite{Vo2}, \cite{Vo3}.
Analogous statements for $\beta=+\infty$ and fixed $N$ are proven in \cite{Marcus}.
Heuristically, one
would like to be able to claim that the $\beta$-independence of the expectations of Theorem
\ref{Theorem_independence} should imply the $\beta$-independence of the
$N\to\infty$ limit for $\a\boxplus_{\beta} \b$ and $\a\boxtimes_{\beta} \b$;
however, we do not know how to produce a rigorous proof along this line, and the
problem of $N\to\infty$ limits for general values of $\beta>0$ remains open.

The global Gaussian fluctuations as $N\to\infty$ for $\a\boxplus_{\beta} \b$ and
$\a\boxtimes_{\beta} \b$ has been addressed for $\beta=2$ in a free probability context using ``second order freeness'' \cite{MiS}, \cite{MiSS}.

\bigskip

There are additional questions concerning the asymptotics as $N\to\infty$. One recent topic is the
study of the global fluctuations for the difference between two adjacent levels in the
corners processes (again at $\beta=1,2,4$) \cite{Kerov}, \cite{Bufetov_LLN},
\cite{ED}, \cite{GZ}, \cite{Sodin}.
The link with taking derivatives of a
polynomial that appears in Theorem \ref{Theorem_crystal} is somewhat visible in the results of
\cite{Sodin}.

Another popular theme is to study \emph{local limits}, e.g.\ the spacings between
individual particles (eigenvalues) in the bulk. At $\beta=2$ the local limits of the
corners--processes were shown to coincide with the sine process in \cite{Metcalfe}, and
similar local asymptotics for $\a \boxplus_{\beta=2} \b$ was proven in
\cite{CheLandon}. At general values of $\beta>0$ one would expect an appearance of
the $\beta$--Sine process, but this is not proven rigorously. What is  known, is
that as $\beta\to\infty$ the $\beta$--Sine processes crystallize on a lattice, see
\cite{Sandier_Serfaty}, \cite{Leble_Serfaty}, and indeed for $\beta=\infty$ version
of the corners process such limit theorems leading to a lattice are known, see
\cite{Farmer-Rhoades}.

\bigskip

The rest of the article is organized as follows. In Section \ref{Section_operations}
we give formal definitions of the operations $\pi^{N\to k}_\beta(\a)$,
$\a\boxtimes_\beta \b$, $\a\boxplus_\beta\b$. Proofs of Theorems
\ref{Theorem_independence}, \ref{Theorem_crystal}, \ref{Theorem_crystallization} are
given in Section \ref{Section_proofs}. Section \ref{Section_Discrete} discusses
extensions to discrete $(q,t)$--setting related to Macdonald polynomials.

\bigskip

{\bf Acknowledgements.}
V.G.~is partially supported by the NSF grants DMS-1407562,
DMS-1664619, by the Sloan Research Fellowship, and by The Foundation Sciences
Math\'{e}matiques de Paris. A. W. M.~is partially supported by NSF grant
DMS-1552520. The authors would like to thank P.~Br\"{a}nd\'{e}n, A.~Borodin, G.~Olshanski, and S.~Sodin
for useful discussions. The authors thank the referee for the comments and suggestions.

\section{Special functions and operations}
\label{Section_operations}

We use several classes of symmetric functions of representation--theoretic origin,
which are degenerations of Macdonald polynomials. We refer to \cite{M} for general
information about these polynomials and only summarize here the required facts.

Let $\Lambda_N$ denote the algebra of symmetric polynomials in $N$--variables
$x_1,\dots,x_N$. Define a difference operator $D_{q,t}$ acting in $\Lambda_N$
through
$$
 D_{q,t} f = \sum_{i=1}^N \left[\prod_{j\ne i} \frac{t x_i-x_j}{x_i-x_j} \right]
 T_{q,x_i},
$$
where $T_{q,x_i}$ is the $q$--shift operator acting as
$$
 [T_{q,x_i} f](x_1,\dots,x_N)=f(x_1,\dots,x_{i-1}, q x_i, x_{i+1},\dots,x_N).
$$
$D_{q,t}$ possesses a complete set of eigenfunctions in $\Lambda_N$, which are the
Macdonald polynomials $P_\lambda(x_1,\dots,x_N;q,t)$ parameterized by $N$--tuples of
non-negative integers $\lambda=(\lambda_1\ge\dots\ge \lambda_N)$, which are called
non-negative \emph{signatures} (of rank $N$):
\begin{equation}
\label{eq_Macdonald_difference}
 D_{q,t} P_\lambda(\cdot;q,t)= \left[\sum_{i=1}^N q^{\lambda_i} t^{N-i} \right]
 P_\lambda(\cdot;q,t).
\end{equation}
$P_\lambda(\cdot;\,q,t)$ is a homogeneous polynomial of  degree
$|\lambda|:=\lambda_1+\dots+\lambda_N$. The leading monomial (with respect to the
lexicographic ordering) of $P_\lambda(\cdot;q,t)$ is $x_1^{\lambda_1}
x_2^{\lambda_2}\cdots x_N^{\lambda_N}$. We normalize the polynomials by declaring
the coefficient of this monomial to be $1$.

The definition \eqref{eq_Macdonald_difference} readily implies the shift property
$$
 P_{\lambda+1}(x_1,\dots,x_N;\,q,t)=(x_1 x_2\cdots x_N)
 P_{\lambda}(x_1,\dots,x_N;\,q,t), \quad
 \lambda+1=(\lambda_1+1,\lambda_2+1,\dots,\lambda_N+1),
$$
which allows one to extend the definition of $P_\lambda$ to general
signatures $\lambda$ with (possibly) negative integer coordinates $\lambda_i$. The
result is a Laurent polynomial.

Another property that links negative and positive signatures is
\begin{equation}
\label{eq_Mac_minus_1}
 P_{\lambda}(x_1^{-1},\dots,x_N^{-1};\,q,t)=P_{-\lambda}(x_1,\dots,x_N;\,q,t), \quad
 -\lambda=(-\lambda_N,-\lambda_{N-1},\dots,-\lambda_1).
\end{equation}
The property \eqref{eq_Mac_minus_1} is somewhat hard to locate in standard references in the
explicit form, but it readily follows from the combinatorial formula for Macdonald polynomials
and \cite[Example VI.6.2 (b)]{M}.

If $q=t$, then Macdonald polynomials coincide with Schur polynomials, and so can be given by an
explicit determinantal formula
$$
 P_\lambda(x_1,\dots,x_N;\, q,q)=s_\lambda(x_1,\dots,x_N)=\frac{\det\bigl[
 x_i^{\lambda_j+N-j}\bigr]_{i,j=1}^N}{\prod_{1\le i<j\le N}(x_i-x_j)}.
$$

We mostly do not use Macdonald polynomials directly, instead relying on three
different degenerations. The Jack symmetric polynomials $J_\lambda$ are defined for
$\theta>0$ through (cf.\ \cite[Chapter VI, Section 10]{M}):
$$
 J_\lambda(x_1,\dots,x_N;\, \theta)=\lim_{q\to 1}P_\lambda(x_1,\dots,x_N;
 \, q,q^{\theta}).
$$
The Heckman--Opdam hypergeometric (for root system of type A, cf.\ \cite{HO},
\cite{Op}, \cite{HS}) functions $\HO_r$ are defined for $\theta>0$ and $N$--tuples
of \emph{distinct real} labels $r=(r_1> r_2>\dots>r_N)$ through (cf.\ \cite[Section
6]{BG_GFF}, \cite{YiSun_HO} and references therein)
\begin{equation}
\label{eq_HO_limit_params}
 t=q^{\theta},\quad q=\exp(-\eps),\quad \lambda = \lfloor \eps^{-1}
 (r_1,\dots,r_N) \rfloor ,\quad
 x_i=\exp(\eps y_i),
\end{equation}
$$
 \HO_r(y_1,\dots,y_N;\,\theta)= \lim_{\eps\to 0} \eps^{\theta N(N-1)/2}
 P_\lambda(x_1,\dots,x_N;\,q,t).
$$
The multivariate Bessel functions $\B_r$ (cf.\ \cite{Dunkl}, \cite{Opdam},
\cite{Jeu}, \cite{Ok_Olsh_shifted_Jack}, \cite{GK}) are defined for $\theta>0$ and
$N$--tuples of distinct real labels $r=(r_1> r_2>\dots>r_N)$ through
$$
 \B_r(z_1,\dots,z_N;\,\theta)=\lim_{\eps\to 0} \HO_{\eps r} (\eps^{-1}
 z_1,\dots,\eps^{-1} z_N;\, \theta).
$$
Or, equivalently (cf.\ \cite[Section 4]{Ok_Olsh_shifted_Jack}),
\begin{equation}
\label{eq_Bessel_limit_params} \lambda = \lfloor \eps^{-1}
 (r_1,\dots,r_N) \rfloor ,\quad
 x_i=\exp(\eps z_i),
\end{equation}
$$
 \B_r(z_1,\dots,z_N;\,\theta)=\lim_{\eps\to 0} \eps^{\theta N(N-1)/2}
 J_\lambda(x_1,\dots,x_N;\,\theta).
$$
We also need the \emph{normalized} versions of all these symmetric functions:
\begin{eqnarray*}
 \hat
 P_\lambda(\cdot;\,q,t)=\frac{P_\lambda(\cdot;,q,t)}{P_\lambda(1,t,\dots,t^{N-1};\,q,t)},
 & & \hat J_\lambda(\cdot;\, \theta)=\frac{J_\lambda(\cdot;\,
 \theta)}{J_\lambda(1,1,\dots,1;\, \theta)},\\
 \hat \HO_r(\cdot;\, \theta)=\frac{\HO(\cdot;\,
 \theta)}{\HO_r(0,-\theta,\dots,(1-N)\theta;\, \theta)},&& \hat
 \B_r(\cdot;\,\theta)=\frac{\B_r(\cdot;\,\theta)}{\B_r(0,\dots,0;\, \theta)}.
\end{eqnarray*}
One advantage of the normalized functions $\hat \HO$ and $\hat \B$ is that one can
now extend their definition to labels $r=(r_1\ge \dots\ge r_N)$ with possibly
coinciding coordinates through limit transitions.

The following property is the label--variable symmetry of the Macdonald (Laurent)
polynomials, see \cite[Chapter VI, Section 6]{M}:
\begin{equation}
\label{eq_Macd_symmetry}
 \hat P_\mu(q^{\lambda_1} t^{N-1},q^{\lambda_2} t^{N-2},\dots,q^{\lambda_N};\,
 q,t) =  \hat P_\lambda(q^{\mu_1} t^{N-1},q^{\mu_2} t^{N-2},\dots,q^{\mu_N};\,
 q,t).
\end{equation}
It implies the following two symmetries for the degenerations:
\begin{equation}
\label{eq_HO_symmetry} \hat
\HO_r(-\mu_1-(N-1)\theta,-\mu_2-(N-2)\theta,\dots,-\mu_N;\, \theta)=\hat
J_\mu(\exp(-r_1),\exp(-r_2),\dots,\exp(-r_N);\, \theta)
\end{equation}
\begin{equation}
\label{eq_Bessel_symmetry} \hat \B_r(\ell_1,\dots,\ell_N;\,\theta)=\hat
\B_\ell(r_1,\dots,r_N;\, \theta).
\end{equation}

\bigskip

The Bessel function $\hat \B_r$ coincides with a (partial) Laplace tranform of the
$\beta=2\theta$--corners process $\{x_i^k\}_{1\le i \le k \le N}$ with top row
$r=(r_1,\dots,r_N)$ of Definition \ref{def_betacorner}, given by
\begin{equation}
\label{eq_Bessel_combinatorial}
 \hat \B_{r}(z_1,\dots,z_N;\,\theta)= \E_{\{x^k_i\}} \left[\exp\left(\sum_{k=1}^{N} z_k
 \cdot \left(\sum_{i=1}^{k} x_i^k-\sum_{j=1}^{k-1} x_j^{k-1}\right)  \right) \right].
\end{equation}
This connection is known as the \emph{combinatorial formula} for Bessel functions;
\eqref{eq_Bessel_combinatorial} is a limit of similar combinatorial formulas for
Macdonald and Jack polynomials, Heckman--Opdam hypergeometric functions.

\bigskip

We will now discuss the connection of the degenerations of Macdonald
polynomials to the addition and multiplication of random matrices. For that we
define four types of connection coefficients, which are the
variants/generalizations of the \emph{Littlewood--Richardson coefficients}. They are
 determined by the following decompositions

\begin{eqnarray}
 \label{eq_LR_Mac}\hat P_\lambda(\cdot;\,q,t) \hat P_\mu(\cdot;\, q,t)&=&\sum_{\nu} c_{\lambda,\mu}^{\nu}(\hat P;\,
 q,t) \hat P_\nu(\cdot;\,q,t),\\
 \label{eq_LR_Jack} \hat J_\lambda(\cdot;\,\theta) \hat J_\mu(\cdot;\, \theta)&=&\sum_{\nu} c_{\lambda,\mu}^{\nu}(\hat J;\,
 \theta) \hat J_\nu(\cdot;\, \theta),\\
 \label{eq_LR_HO}\hat \HO_r(\cdot;\,\theta) \hat \HO_\ell(\cdot;\,
\theta)&=&\int_s c_{\ell,r}^{s}(\hat \HO;\,
 \theta) \hat \HO_s(\cdot;\, \theta)\, ds,\\
 \label{eq_LR_Bes} \hat \B_r(\cdot;\,\theta) \hat \B_\ell(\cdot;\,
\theta)&=&\int_{s} c_{\ell,r}^{s}(\hat \B;\,
 \theta) \hat \B_s(\cdot;\, \theta)\, ds.
 \end{eqnarray}
The meaning of the decomposition in \eqref{eq_LR_Mac}, \eqref{eq_LR_Jack} is
straightforward, as this is an expansion of a symmetric (Laurent) polynomial on the
left--hand side by functions forming a linear basis of the space of all such
polynomials. The sums in \eqref{eq_LR_Mac}, \eqref{eq_LR_Jack} have only finitely
many non-zero terms; in more details, the coefficients $c_{\lambda,\mu}^{\nu}(\hat
P;\, q,t)$ and $c_{\lambda,\mu}^{\nu}(\hat J;\, \theta)$ are non-zero only if
\begin{equation}
\label{eq_LR_support} \nu_1+\dots+\nu_N=(\lambda_1+\mu_1)+\dots+(\lambda_N+\mu_N),\
\quad \lambda_N+\mu_N \le \nu_i \le \lambda_1+\mu_1,\, 1\le i \le N.
\end{equation}

Although, in principle, $c_{\lambda,\mu}^{\nu}(\hat J;\,
 \theta)$, $c_{\ell,r}^{s}(\hat \HO;\,
 \theta)$, and $c_{\ell,r}^{s}(\hat \B;\,
 \theta)$ can be all obtained from $c_{\lambda,\mu}^{\nu}(\hat P;\,
 \theta)$ by limit transitions, the exact mathematical
 meaning of \eqref{eq_LR_HO}, \eqref{eq_LR_Bes} is a bit more delicate. There is a
solid amount of literature devoted to the expansions of functions in appropriate
spaces into the integrals of $\HO$ and $\B$ functions; this is a far--reaching
generalization of the conventional Fourier transform, and it is known under the name
Cherednik transform for $\HO$ and Dunkl transform for $\B$. We refer to \cite{Anker}
for a recent brief review with many references to the original articles and more
detailed treatments. In particular, it is rigorously known that
$$
 f\mapsto \int_{s} c_{\ell,r}^{s}(\hat \HO;\,
 \theta) \cdot f\, ds, \quad  f\mapsto \int_{s} c_{\ell,r}^{s}(\hat \B;\,
 \theta) \cdot f\, ds,
$$
are well-defined as distributions with compact support. The support of
$c_{\ell,r}^{s}(\hat \HO;\,
 \theta)$, $c_{\ell,r}^{s}(\hat \B;\,
 \theta)$, is a subset of those $s$ which satisfy
\begin{equation}
\label{eq_LR_cont_support} s_1+\dots+s_N=(r_1+\ell_1)+\dots+(r_N+\ell_N),\ \quad
r_N+\ell_N \le s_i \le r_1+\ell_1,\, 1\le i \le N.
\end{equation}
 However, much more is
conjectured, as we discuss below.

\begin{conjecture}\label{Conjecture_positivity} All the coefficients $c_{\lambda,\mu}^{\nu}(\hat P;\,
 q,t)$ for $0<q,t<1$, (and hence also $c_{\lambda,\mu}^{\nu}(\hat J;\,
 \theta)$, $c_{\ell,r}^{s}(\hat \HO;\,
 \theta)$, $c_{\ell,r}^{s}(\hat \B;\,
 \theta)$ for $\theta>0$) are non-negative.
\end{conjecture}
 In the $q=t$ case, $c_{\lambda,\mu}^{\nu}(\hat P;\,
 q,q)$ are (up to a normalization) conventional Littlewood--Richardson coefficients,
 which are known to be non-negative due to either representation--theoretic interpretations
 or combinatorial formulas, see e.g.\ \cite[Chapter I, Section 10]{M}. For $\theta=1/2,1,2$, the coefficients $c_{\lambda,\mu}^{\nu}(\hat J;\,
 \theta)$, $c_{\ell,r}^{s}(\hat \HO;\,
 \theta)$, $c_{\ell,r}^{s}(\hat \B;\,
 \theta)$ are again known to be non-negative due to representation--theoretic
 interpretations, cf.\ \cite[Chapter VII]{M}.
 When $q=0$, the non-negativity of the coefficients $c_{\lambda,\mu}^{\nu}(\hat P;\,
 q,t)$ are due to known combinatorial formulas, see \cite[Theorem 4.9]{Ram}, \cite[Theorem
1.3]{Sc}, \cite{KM} and references therein. The
 non-negativity of $c_{\lambda,\mu}^{\nu}(\hat J;\,
 \theta)$ would also follow from a (still open) conjecture of Stanley \cite[Conjecture 8.3]{Stanley}.
 We refer to \cite{Rosler} for progress concerning the non-negativity of $c_{\ell,r}^{s}(\hat \B;\,
 \theta)$.

\bigskip

Since our definitions of $\hat P$, $\hat J$, $\hat \HO$, $\hat B$ imply that
$$
 \sum_{\nu} c_{\lambda,\mu}^{\nu}(\hat P;\,
 q,t)= \sum_{\nu} c_{\lambda,\mu}^{\nu}(\hat J;\,
 \theta)= \int_{s} c_{\ell,r}^{s}(\hat \HO;\,
 \theta)ds= \int_{s} c_{\ell,r}^{s}(\hat \B;\,
 \theta)ds=1,
$$
Conjecture \ref{Conjecture_positivity} would then allow for these coefficients (with fixed $\lambda,\mu$ or $r,\ell$, and varying $\nu$ or $s$) can be treated as
\emph{probability measures}. Without Conjecture \ref{Conjecture_positivity} the
coefficients in \eqref{eq_LR_Mac}, \eqref{eq_LR_Jack} define \emph{signed} measures
of total mass $1$, while the coefficients in \eqref{eq_LR_HO}, \eqref{eq_LR_Bes}
have only distributional meaning.

The connection of these definitions to random matrix operations is explained in the
next two propositions. For $r=(r_1,\dots,r_N)$, let $\exp(r)$ denote the vector
$(\exp(r_1),\dots,\exp(r_N))$
\begin{proposition}\label{Prop_muliplication_HO}
 Choose $\theta=\frac{1}{2}$, $1$, or $2$, and let $\beta=2\theta$. Take $r=(r_1,\dots,r_N)$, $\ell=(\ell_1,\dots,\ell_N)$ and let
 $s=(s_1,\dots,s_N)$ be $c_{r,\ell}^s(\hat \HO;\, \theta)$--distributed random
 vector. Then in the notations of Section \ref{Section_intro}
\begin{equation}
 \exp(s) \stackrel{d}{=} \exp(r) \boxtimes_{\beta} \exp(\ell).
\end{equation}
\end{proposition}
\begin{proof}
 Combining \eqref{eq_LR_HO} with \eqref{eq_HO_symmetry} we conclude that for $c_{r,\ell}^s(\hat \HO;\, \theta)$--distributed random
 vector and each integral vector $\lambda=(\lambda_1\ge\lambda_2\ge\dots\ge\lambda_N)$, we have
 \begin{multline*}
  \E \hat J_\lambda(\exp(-s_1),\dots,\exp(-s_N);\, \theta)\\=
  \hat J_\lambda(\exp(-r_1),\dots,\exp(-r_N);\, \theta) \hat J_\lambda(\exp(-\ell_1),\dots,\exp(-\ell_N);\, \theta)
 \end{multline*}
 Taking into account the Jack limit of \eqref{eq_Mac_minus_1}, we can equivalently
 write
 \begin{equation}
 \label{eq_Expecation_Jack}
  \E \hat J_\lambda(\exp(s_1),\dots,\exp(s_N);\, \theta)=
  \hat J_\lambda(\exp(r_1),\dots,\exp(r_N);\, \theta) \hat J_\lambda(\exp(\ell_1),\dots,\exp(\ell_N);\,
  \theta).
 \end{equation}
 Note that $\hat J_\lambda$ form a basis in symmetric Laurent polynomials, and
 therefore, the expectations of the form \eqref{eq_Expecation_Jack} uniquely define
 the distribution of $\exp(s_1),\dots,\exp(s_N)$. On the other hand, the identity
 \eqref{eq_Expecation_Jack} is well-known to hold with $\exp(s)$ replaced by
 $\exp(r) \boxtimes_{\beta} \exp(\ell)$, see \cite[Chapter VII]{M}, \cite[Section
 13.4.3]{For} and references therein.
\end{proof}

\begin{proposition}\label{Prop_addition_B}
 Choose $\theta=\frac{1}{2}$, $1$, or $2$, and let $\beta=2\theta$. Take $r=(r_1,\dots,r_N)$, $\ell=(\ell_1,\dots,\ell_N)$ and let
 $s=(s_1,\dots,s_N)$ be $c_{r,\ell}^s(\hat \B;\, \theta)$--distributed random
 vector. Then in the notations of Section \ref{Section_intro}
\begin{equation}
\label{eq_x6}
 s \stackrel{d}{=} r \boxplus_{\beta} \ell.
\end{equation}
\end{proposition}
\begin{proof}
  The formula \eqref{eq_Bessel_combinatorial} identifies the Bessel functions with
  Laplace transforms of the $\beta$--corners processes. At $\theta=1/2,1,2$, this
  implies an identification with Laplace transforms of real/complex/quaternion
  random matrices, see the discussion after Definition \ref{def_betacorner}. Laplace transform of the sum of independent random variables
  (matrices) is a product of Laplace transforms, hence the definition of $r \boxplus_{\beta} \ell$ coincides
  with that of $c_{r,\ell}^s(\hat \B;\,
  \theta)$--distributed $s$.
\end{proof}

Propositions \ref{Prop_muliplication_HO}, \ref{Prop_addition_B} motivate
extrapolation of matrix operations to general $\beta>0$.

\begin{definition}
\label{Def_multiplication}
  Choose $\theta>0$, and let $\beta=2\theta$. Take $\a=(a_1,\dots,a_N)$,
  $\b=(b_1,\dots,b_N)$ and assume $\a=\exp(r)$, $\b=\exp(\ell)$. Let
 $s=(s_1,\dots,s_N)$ be $c_{r,\ell}^s(\hat \HO;\, \theta)$--distributed random
 vector. Then in the notations of Section \ref{Section_intro}
\begin{equation}
\label{eq_x7}
  \a \boxtimes_{\beta} \b:=\exp(s).
\end{equation}
\end{definition}
\begin{definition}\label{Def_addition}
 Choose $\theta>0,$ and let $\beta=2\theta$. Take $\a=(a_1,\dots,a_N)$, $\b=(b_1,\dots,b_N)$ and let
 $s=(s_1,\dots,s_N)$ be $c_{\a,\b}^s(\hat \B;\, \theta)$--distributed random
 vector. Then in the notations of Section \ref{Section_intro}
\begin{equation}
\label{eq_x8}
 \a \boxplus_{\beta} \b:=s.
\end{equation}
\end{definition}
\begin{definition}\label{Def_projections}
 Choose $\theta>0,$ and let $\beta=2\theta$. Take $\a=(a_1,\dots,a_N)$, $k\le N$ and
 let $s=(s_1,\dots,s_k)$ be the $k$th row of $\beta$--corners process with top row $\a$, i.e.\
 the particles $x_1^k,x_2^k,\dots,x_k^k$ in the notations of Definition \ref{def_betacorner}.  Then
\begin{equation}
\label{eq_x9} \pi^{N\to k}_\beta(\a):=s.
\end{equation}
\end{definition}
Comparing with \eqref{eq_Bessel_combinatorial}, we see the following property (which
can be taken as an equivalent definition of $\pi^{N\to k}_\beta(\a)$):
\begin{equation}
\label{eq_Projection_Bessel}
 \E \B_{\pi^{N\to k}_\beta(\a)} (z_1,\dots,z_k;\, \beta/2)= \B_{\a}(z_1,\dots,z_k,
 \underbrace{0,\dots,0}_{N-k}\,).
\end{equation}

\section{Proofs}

\label{Section_proofs}

\subsection{The expectation identities} \label{Section_expectations} The proofs of Theorems
\ref{Theorem_independence}, \ref{Theorem_crystal} are based on the evaluations of
the expectations of Jack polynomials summarized in the next three propositions.

For a finite integer vector $\lambda=(\lambda_1\ge\lambda_2\ge \dots\ge 0)$, we
define
 $$
 (t)_{\lambda;\theta}=\prod_{\begin{smallmatrix}(i,j)\in \mathbb Z_{>0} \times \mathbb Z_{>0}\\ j\le
 \lambda_i \end{smallmatrix}} \bigl(t+(j-1)-\theta(i-1) \bigr).
$$
 We need to introduce yet another normalization of Jack polynomials.
The dual Jack polynomials $ J^{\mathrm{dual}}_\lambda$ differ from $J_\lambda$ by
multiplication by explicit constants independent of $N$ and $x_1,\dots,x_N$ (see
\cite[Chapter VI, Section 10]{M}) in such a way that makes the following identity true:
$$
 \sum_{\lambda_1\ge\dots\ge\lambda_N\ge 0} J_\lambda(x_1,\dots,x_N;\, \theta)
 J^{\mathrm{dual}}_\lambda(y_1,\dots,y_N;\, \theta)=\prod_{i,j=1}^N (1-x_i y_j)^{-\theta}.
$$
We also need corresponding Jack--Littlewood--Richardson coefficients
$c^{\lambda}_{\mu,\nu}(J^{\mathrm {dual}},\, \theta)$, defined by
$$
 J^{\mathrm {dual}}_\mu(x_1,\dots,x_N;\, \theta) \cdot J^{\mathrm {dual}}_\nu(x_1,\dots,x_N;\,\theta)=\sum_{\lambda}
  c^{\lambda}_{\mu,\nu}(J^{\mathrm {dual}};\,\theta)\cdot J^{\mathrm {dual}}_\lambda(x_1,\dots,x_N;\, \theta).
$$

\begin{proposition} \label{Proposition_Jack_observable_product} Fix positive vectors of eigenvalues
$\a$ and $\b$, and let  $\a \boxtimes_\beta \b$ denote the corresponding
$N$--dimensional random vector in the notations of Section \ref{Section_intro}.
 For each $\lambda=\lambda_1\ge\lambda_2\ge\dots\ge\lambda_N\ge 0$ we have
\begin{equation} \label{eq_Jack_observable_product}
 \E J_\lambda( \a \boxtimes_\beta \b;\, \beta/2)=
 \frac{J_\lambda( \a ;\, \beta/2) \cdot J_\lambda( \b ;\,
 \beta/2)}{J_\lambda(1,\dots,1;\, \beta/2)}.
\end{equation}
\end{proposition}
\begin{proof}
 This is equivalent to \eqref{eq_Expecation_Jack}, whose proof, in fact, did not use
 $\theta=1/2,1,2$.
\end{proof}

\begin{proposition} \label{Proposition_Jack_observable_projection} Fix $k\le N$ and $\a$, and let
$\pi^{\beta}_{N\to k}(\a)$ denote the corresponding $k$--dimensional random vector
 in the notations of Section
\ref{Section_intro}.
 For each $\lambda=\lambda_1\ge\lambda_2\ge\dots\ge\lambda_k\ge 0$ we have
\begin{equation}\label{eq_Jack_observable_projection}
 \E  J_\lambda( \pi^{\beta}_{N\to k}(\a);\, \beta/2)= J_\lambda(\a;\ \beta/2)\cdot \frac{
  (k\beta/2)_{\lambda;\beta/2}}{(N\beta/2)_{\lambda;\beta/2}},
\end{equation}
 where in the right--hand side we view $\lambda$ as a signature of rank $N$ by
 adding $N-k$ zeros.
\end{proposition}
\begin{proof}
 We use the series expansion of the Bessel functions, see \cite[Section 4]{Ok_Olsh_shifted_Jack}:
\begin{multline}
\label{eq_Bessel_binomial}
 \B_{a_1,\dots,a_N}(x_1,\dots,x_N;\, \beta)=\sum_{\lambda=(\lambda_1\ge\dots\ge\lambda_N\ge 0)}
 \frac{J_\lambda(x_1,\dots,x_N;\, \beta/2) J^{\mathrm{dual}}_\lambda(a_1,\dots,a_N;\,
 \beta/2)}{(N\beta/2)_{\lambda; \beta/2}}.
\end{multline}
 We plug \eqref{eq_Bessel_binomial} into both sides of \eqref{eq_Projection_Bessel} and use the
 stability property of Jack polynomials valid for all $\lambda=(\lambda_1\ge\dots\ge\lambda_N\ge
 0)$:
 $$
  J_{\lambda}(z_1,\dots,z_k, \underbrace{0,\dots,0}_{N-k}\,;\theta)=\begin{cases}
  J_{\lambda}(z_1,\dots,z_k\,;\theta),&
  \lambda_{k+1}=\dots=\lambda_{N}=0,\\
  0,&\text{otherwise,}
  \end{cases}
 $$
 where in the right--hand side we treat $\lambda$ as a signature of rank $k$ by
 removing $N-k$ zero coordinates. We get
\begin{multline} \label{eq_x4}
  \E
  \sum_{\lambda=(\lambda_1\ge\dots\ge\lambda_k\ge 0)}  \frac{J_\lambda(\pi^{N\to k}_\beta(\a);\, \beta/2)
  J^{\mathrm{dual}}_\lambda(z_1,\dots,z_k;\,
 \beta/2)}{(k\beta/2)_{\lambda; \beta/2}},
\\= \sum_{\lambda=(\lambda_1\ge\dots\ge\lambda_k\ge 0)}
 \frac{J_\lambda(\a;\, \beta/2) J^{\mathrm{dual}}_\lambda(z_1,\dots,z_k;\,
 \beta/2)}{(N\beta/2)_{\lambda; \beta/2}},
\end{multline}
Comparing the coefficient of $J^{\mathrm{dual}}_\lambda(z_1,\dots,z_k;\, \beta/2)$
in both sides of \eqref{eq_x4} we arrive at the desired statement.
\end{proof}

\begin{proposition} \label{Proposition_Jack_observable_sum} Fix $\a$ and $\b$,
and let  $\a \boxplus_\beta \b$ denote the corresponding $N$--dimensional random
vectors (in $k$, $N$, $N$ dimensional spaces, respectively) in the notations of
Section \ref{Section_intro}.
 For each $\lambda=\lambda_1\ge\lambda_2\ge\dots\ge\lambda_N\ge 0$ we have
\begin{equation}\label{eq_Jack_observable_sum}
 \E  J_\lambda( \a \boxplus_\beta \b;\, \beta/2)= (N\beta/2)_{\lambda;\beta/2} \sum_{\nu,\mu}
  c^{\lambda}_{\mu,\nu}(J^{\mathrm {dual}};\, \beta/2)\cdot
  \frac{J_\mu(\a;\ \beta/2)}{(N\beta/2)_{\mu;\beta/2}} \cdot   \frac{J_\nu(\b;\
  \beta/2)}{(N\beta/2)_{\nu;\beta/2}}.
\end{equation}
\end{proposition}
\begin{remark} \label{Remark_JLR}
The sum in \eqref{eq_Jack_observable_sum} is finite, since
$c^{\lambda}_{\mu,\nu}(J^{\mathrm {dual}};\, \beta/2)$ vanishes unless $0\le \mu_i
\le \lambda_i$, $0\le \nu_i \le \lambda_i$, $\sum_{i=1}^N
(\mu_i+\nu_i-\lambda_i)=0$.
\end{remark}
\begin{proof}[Proof of Proposition \ref{Proposition_Jack_observable_sum}]
 The proof is similar to that of Proposition
 \ref{Proposition_Jack_observable_projection}. We rewrite \eqref{eq_LR_HO} as
 $$
  \E \B_{\a \boxplus_\beta \b}(z_1,\dots,z_N;\, \beta/2)=\B_{\a}(z_1,\dots,z_N;\, \beta/2)\cdot \B_{\b }(z_1,\dots,z_N;\,
  \beta/2),
 $$
 then plug \eqref{eq_Bessel_binomial} into both sides, and compare the coefficients
 of $J^{\mathrm{dual}}_\lambda(z_1,\dots,z_N;\, \beta/2)$.
\end{proof}

\subsection{Proof of Theorem \ref{Theorem_independence}}

\label{Section_Proof_of_Th_independence}

We use Propositions \ref{Proposition_Jack_observable_product},
\ref{Proposition_Jack_observable_projection}, \ref{Proposition_Jack_observable_sum}
with a particular choice of $\lambda$: $\lambda_1=\dots=\lambda_\ell=1$,
$\lambda_{\ell+1}=\lambda_{\ell+2}=\dots=0$. For such $\lambda$, Jack polynomial
coincides with elementary symmetric function,  cf.\ \cite[Chapter VI]{M}:
$$
 J_{1^{\ell},0^{N-\ell}}(x_1,\dots,x_N;\, \theta)=e_{\ell}(x_1,\dots,x_N).
$$
Then \eqref{eq_Jack_observable_product} becomes
\begin{equation}
\label{eq_x10}
 \E e_{\ell}(\a \boxtimes_{\beta}\b)=\frac{e_{\ell}(\a) e_\ell(\b)}{{N \choose \ell}}.
\end{equation}
On the other hand, evaluating the coefficient of $z^{N-\ell}$,
\eqref{eq_ff_multiply} is equivalent to
\begin{equation} \label{eq_x11} \E
e_{\ell}(\a \boxtimes_{\beta}\b)=\frac{1}{N!} \sum_{\sigma\in S_N} e_{\ell}(a_1
b_{\sigma(1)}, a_2 b_{\sigma(2)},\dots,a_N b_{\sigma(N)}).
\end{equation}
It is straightforward to check that the right--hand sides of \eqref{eq_x10} and
\eqref{eq_x11} are the same.

\smallskip

Further, \eqref{eq_Jack_observable_projection} becomes
\begin{equation}
\label{eq_x12}
 \E e_{\ell}( \pi^{N\to k}_\beta(\a))=e_{\ell}(\a) \cdot \frac{k(k-1)\cdots
 (k-\ell+1)}{N(N-1)\cdots (N-\ell+1)},
\end{equation}
which is the same expression as the coefficient of $z^{k-\ell}$ in
\eqref{eq_ff_projection}.

\smallskip

 We finally apply  Proposition \ref{Proposition_Jack_observable_sum}
 with $\lambda_1=\lambda_2\dots=\lambda_\ell=1$,
 $\lambda_{\ell+1}=\dots=\lambda_N=0$. Using Remark \ref{Remark_JLR} we conclude
 that the right--hand side has $\ell+1$ term and the $(p+1)$st term has
 $\mu_1=\mu_2\dots=\mu_p=1$,
 $\mu_{p+1}=\dots=\mu_N=0$;  $\nu_1=\nu_2\dots=\nu_{\ell-p}=1$,
 $\nu_{\ell-p+1}=\dots=\nu_N=0$. We get
\begin{multline}
\label{eq_elementary_expectation_product}
  \E e_\ell(\a \boxplus_\beta \b)=N(N-1)\cdots(N-\ell+1)\sum_{p=0}^{\ell}
  c^{1^{\ell}}_{1^{p},1^{\ell-p}}( J^{\mathrm{dual}};\, \beta/2)
  \\ \times \frac{e_p(\a)}{N(N-1)\cdots(N-p+1)} \cdot   \frac{e_{\ell-p}(\b)}{
  N(N-1)\cdots(N-\ell+p+1)}.
\end{multline}
It remains to find the value of the constant $ c^{1^{\ell}}_{1^{p},1^{\ell-p}}(
J^{\mathrm{dual}};\, \beta/2)$. For that we use the automorphism $\omega_{\theta}$
of the algebra of symmetric functions in \emph{infinitely many variables} (cf.\
\cite[Chapter VI, Section 10]{M}) with $\theta=\beta/2$. It has the following action
on $e_\ell$:
$$
 \omega_\theta\bigl( J^{\mathrm{dual}}_{1^{\ell}}(\cdot;\, \theta)\bigr)=J_{(\ell,0,0,\dots)}(\cdot;\,
 \theta^{-1}),
$$
and therefore $ c^{1^{\ell}}_{1^{p},1^{\ell-p}}( J^{\mathrm{dual}};\, \beta/2)$ is
the coefficient of $\lambda=(\ell,0,0,\dots)$ in the expansion
\begin{equation}
\label{eq_product_complete}
 J_{(p,0,\dots)}(\cdot;\, 2/\beta) \cdot J_{(\ell-p,0,\dots)}(\cdot;\,
 2/\beta)=\sum_{\lambda} c^{\lambda}_{(p,0,\dots),(\ell-p,0,\dots)}(J;\, 2/\beta)
 \cdot J_\lambda(\cdot;\, 2/\beta).
\end{equation}
This coefficient is readily found by comparing the coefficient of the leading
monomial $x_1^{\ell}$ in both sides of \eqref{eq_product_complete}, and therefore
$c^{1^{\ell}}_{1^{p},1^{\ell-p}}( J^{\mathrm{dual}};\, \beta/2)=1$. We conclude that
\begin{equation}
\label{eq_elementary_expectation_product_simple}
  \E e_\ell(\a \boxplus_\beta \b)=\sum_{p=0}^{\ell}
 e_p(\a) e_{\ell-p}(\b)
  \frac{N(N-1)\cdots(N-\ell+1)}{N(N-1)\cdots(N-p+1) \cdot N(N-1)\cdots(N-\ell+p+1)} .
\end{equation}
On the other hand, the coefficient of $z^{N-\ell}$ in \eqref{eq_ff_add} gives
\begin{equation}
\label{eq_x13}
 \E e_\ell(\a \boxplus_\beta \b) = \frac{1}{N!} \sum_{\sigma\in S(N)}
 e_\ell(a_1+b_{\sigma(1)},a_2+b_{\sigma(2)},\dots,a_N+b_{\sigma(N)}).
\end{equation}
One readily checks that the right--hand sides of
\eqref{eq_elementary_expectation_product_simple} and \eqref{eq_x13} give the same
expression.

\subsection{Proof of Theorem \ref{Theorem_crystal}}
\label{Section_Proof_of_Th_crystal_1} The proof is based on the following two limit
relations, which can be found e.g.\ in \cite[Proposition 7.6]{Stanley}\footnote{The
parameter $\alpha$ of \cite{Stanley} is $\theta^{-1}$.}:
  \begin{equation}
  \label{eq_Jack_infty}
   \lim_{\theta\to\infty} J_\lambda(x_1,\dots,x_N;\, \theta)=\prod_{i=1}^{N}
   \left[e_i(x_1,\dots,x_N)\right]^{\lambda_i-\lambda_{i+1}}.
  \end{equation}
  \begin{equation}
  \label{eq_Jack_zero}
   \lim_{\theta\to 0} J_\lambda(x_1,\dots,x_N;\, \theta)=m_\lambda(x_1,\dots,x_N),
  \end{equation}
  where $e_k$ is the elementary symmetric function and $m_\lambda$ is the monomial symmetric function.

\smallskip

We start with $\lim_{\beta\to\infty}\a \boxtimes_{\beta} \b$. Take any collection of
polynomials $f_1,\dots,f_m$ in $N$ variables. We aim to prove that the following
limits exist and satisfy
\begin{equation}
\label{eq_convergence_delta}
 \lim_{\beta\to\infty} \E \Biggl[\prod_{i=1}^m f_i(\a \boxtimes_{\beta} \b) \Biggr]
 =\lim_{\beta\to\infty} \prod_{i=1}^m \Biggl[\E f_i(\a \boxtimes_{\beta} \b)\Biggr].
\end{equation}
This precisely means that $\lim_{\beta\to\infty} \a \boxtimes_\beta \b$ is a delta
function at a certain point, and then Theorem \ref{Theorem_independence} would
identify this point with roots of $Q^{\boxtimes}(z)$. Note that under Conjecture
\ref{Conjecture_positivity}, $\a \boxtimes_\beta b$ is a bona--fide random variable,
and therefore, convergence to a delta--function for polynomial test functions $f_i$
would imply a similar convergence for arbitrary continuous test--functions (without
Conjecture \ref{Conjecture_positivity}, formally, there might be no such
implication).

Since $\a \boxtimes_{\beta} \b$ is supported on \emph{ordered} $N$--tuples of reals,
it suffices to consider \emph{symmetric} polynomials $f_i$ in
\eqref{eq_convergence_delta}. Further, since \eqref{eq_convergence_delta} is
multi-linear in $f_i$, it suffices to check it on an arbitrary basis of the
algebra $\Lambda_N$ of symmetric polynomials. This algebra is generated by elementary
symmetric functions $e_1,\dots,e_N$, and so we can choose functions of the form
\begin{equation}
\label{eq_e_basis}
 e_{(\lambda)}:= e_1^{\lambda_1-\lambda_2} e_2^{\lambda_2-\lambda_3} \cdots e_{N-1}^{\lambda_{N-1}-\lambda_N}
 e_N^{\lambda_N},\quad \lambda_1\ge\lambda_2\ge\dots\ge\lambda_N\ge 0
\end{equation}
as such basis. Therefore, \eqref{eq_convergence_delta} reduces to the statement that
for any $\lambda$:
\begin{equation}
\label{eq_convergence_delta_e}
 \lim_{\beta\to\infty} \E \left[ e_{(\lambda)}(\a \boxtimes_{\beta} \b)\right]
 =\lim_{\beta\to\infty} \prod_{i=1}^N \E \left[ e_i^{\lambda_{i}-\lambda_{i+1}}(\a \boxtimes_{\beta} \b)
 \right].
\end{equation}
We now prove \eqref{eq_convergence_delta_e}. Fix $\lambda$ and decompose
$e_{(\lambda)}$ into a linear combination of Jack polynomials
\begin{equation}
\label{eq_x14}
 e_{(\lambda)}(x_1,\dots,x_N)=\sum_{\mu} T_{\lambda}^{\mu}(\beta/2)
 J_\mu(x_1,\dots,x_N;\, \beta/2).
\end{equation}
Note that the sum in \eqref{eq_x14} is finite, as
$\mu=(\mu_1\ge\mu_2\ge\dots\ge\mu_N\ge 0)$ is required to satisfy $\sum_{i=1}^N
\mu_i=\sum_{i=1}^N\lambda_i$. Further, \eqref{eq_Jack_infty} implies that
$$
 \lim_{\beta\to\infty} T_{\lambda}^{\mu}(\beta/2)=\begin{cases} 1,& \lambda=\mu,\\
 0,&\text{otherwise.} \end{cases}
$$
Thus, using Proposition \ref{Proposition_Jack_observable_product},
\eqref{eq_Jack_infty}, and \eqref{eq_x10} we get
\begin{multline*}
 \lim_{\beta\to\infty} \E \bigl[ e_{(\lambda)}(\a\boxtimes_\beta \b)\bigr]=\lim_{\beta\to\infty}\sum_{\mu}
  T_{\lambda}^{\mu}(\beta/2) \frac{J_\mu(\a;\, \beta/2) J_\mu(\b;\,
 \beta/2)}{J_\mu(1,\dots,1;\,\beta/2)}\\= \frac{\prod_{i=1}^N e_i^{\lambda_{i}-\lambda_{i+1}}(\a) \prod_{i=1}^N
 e_i^{\lambda_{i}-\lambda_{i+1}}(\b)}{\prod_{i=1}^N e_i^{\lambda_{i}-\lambda_{i+1}}(\underbrace{1,\dots,1}_N)}=
 \prod_{i=1}^N \E\bigl[ e_i^{\lambda_{i}-\lambda_{i+1}}(\a\boxtimes_\beta \b)\bigr]=
 \lim_{\beta\to\infty}  \prod_{i=1}^N \E \bigl[
 e_i^{\lambda_{i}-\lambda_{i+1}}(\a\boxtimes_\beta
 \b)\bigr],
\end{multline*}
which proves \eqref{eq_convergence_delta_e}.

\medskip

Next, we investigate $\lim_{\beta\to\infty} \pi^{N\to k}_{\beta}(\a)$. Again it
suffices to prove that for every $\lambda_1\ge\dots\ge\lambda_k\ge 0$,
\begin{equation}
\label{eq_convergence_projection_delta_e}
 \lim_{\beta\to\infty} \E \left[ e_{(\lambda)}(\pi^{N\to k}_{\beta}(\a))\right]
 =\lim_{\beta\to\infty} \prod_{i=1}^k \E \left[ e_i^{\lambda_{i}-\lambda_{i+1}}(\pi^{N\to k}_{\beta}(\a))
 \right].
\end{equation}
Using Proposition \ref{Proposition_Jack_observable_projection},
\eqref{eq_Jack_infty}, and \eqref{eq_x12} we have
\begin{multline*}
 \lim_{\beta\to\infty} \E \left[ e_{(\lambda)}(\pi^{N\to k}_{\beta}(\a))\right]
 =\lim_{\beta\to\infty} \sum_{\mu}
  T_{\lambda}^{\mu}(\beta/2) J_{\mu}(\a;\, \beta/2)
  \frac{(k\beta/2)_{\mu;\beta/2}}{(N\beta/2)_{\mu;\beta/2}}
\\ = \prod_{i=1}^k e_i^{\lambda_{i}-\lambda_{i+1}}(\a)
\prod_{\begin{smallmatrix}(i,j)\in \mathbb
 Z_{>0}\times \mathbb Z_{>0}\\ j\le \lambda_i\end{smallmatrix}} \frac{k-i+1}{N-i+1}
 \\=\prod_{i=1}^k e_i^{\lambda_i-\lambda_{i+1}}(\pi^{N\to k}_\beta(\a))=
 \lim_{\beta\to\infty} \prod_{i=1}^k e_i^{\lambda_i-\lambda_{i+1}}(\pi^{N\to
 k}_\beta(\a)).
\end{multline*}

\medskip

Finally, we turn to $\lim_{\beta\to\infty} \a \boxplus_\beta \b$, and again prove
that for every $\lambda_1\ge\dots\ge\lambda_N\ge 0$
\begin{equation}
\label{eq_convergence_sum_delta_e}
 \lim_{\beta\to\infty} \E \left[ e_{(\lambda)}(\a \boxplus_{\beta} \b)\right]
 =\lim_{\beta\to\infty} \prod_{i=1}^N \E \left[ e_i^{\lambda_{i}-\lambda_{i+1}}(\a \boxplus_{\beta} \b)
 \right].
\end{equation}

Thus, using Proposition \ref{Proposition_Jack_observable_sum} and
\eqref{eq_Jack_infty} we get
\begin{multline*}
 \lim_{\beta\to\infty} \E \bigl[ e_{(\lambda)}(\a\boxplus_\beta \b)\bigr]\\=
  \sum_{\nu,\mu}
  \lim_{\beta\to\infty} \left[ c^{\lambda}_{\mu,\nu}(J^{\mathrm {dual}};\, \beta/2)\cdot
  \frac{(N\beta/2)_{\lambda;\beta/2}}{(N\beta/2)_{\nu;\beta/2} (N\beta/2)_{\mu;\beta/2} }\right]
   \prod_{i=1}^N \left[ e_i^{\mu_{i}-\mu_{i+1}}(\a) e_i^{\nu_i-\nu_{i+1}}(\b) \right].
\end{multline*}
By degree considerations, $c^{\lambda}_{\mu,\nu}(J^{\mathrm {dual}};\, \beta/2)$ is
non-zero only if $\sum_{i} (\mu_i+\nu_i)=\sum_i \lambda_i$. In this case
$$
 \lim_{\beta\to\infty} \frac{(N\beta/2)_{\lambda;\beta/2}}{(N\beta/2)_{\nu;\beta/2} (N\beta/2)_{\mu;\beta/2}
 }= \frac{\prod\limits_{j\le
 \lambda_i} \bigl(N+1-i\bigr)}{\prod\limits_{j\le
 \lambda_i } \bigl(N+1-i \bigr) \cdot \prod\limits_{j\le
 \lambda_i } \bigl(N+1-i \bigr)}=
  \prod_{i=1}^N (N+1-i)^{\lambda_i-\mu_i-\nu_i}
$$
For the value of the constant $  \lim_{\beta\to\infty}
c^{\lambda}_{\mu,\nu}(J^{\mathrm {dual}};\, \beta/2)$ we use the automorphism
$\omega_{\theta}$ of the algebra of symmetric functions in infinitely many variables
(cf.\ \cite[Chapter VI, Section 10]{M}) with $\theta=\beta/2$. It has the following
action on Jack polynomials:
$$
 \omega_\theta\bigl( J^{\mathrm{dual}}_{\lambda}(\cdot;\, \theta)\bigr)=J_{\lambda'}(\cdot;\,
 \theta^{-1}),
$$
where $\lambda'$ is the transpose partition, defined through
$$
 \lambda'_j=|\{i\in \mathbb Z_{>0}:\lambda_i\ge j\}|, \quad j=1,2,\dots.
$$
Then using \eqref{eq_Jack_zero} we conclude
$$
 \lim_{\beta\to\infty} c^{\lambda}_{\mu,\nu}(J^{\mathrm {dual}};\, \beta/2)=
 c^{\lambda'}_{\mu',\nu'}(m),
$$
where $m$ stands for the monomial symmetric functions; that is, the coefficients are
defined from the decomposition
$$
 m_{\mu'} m_{\nu'}=\sum_{\lambda'} c^{\lambda'}_{\mu',\nu'}(m) m_{\lambda'}.
$$
The monomial symmetric functions are easy to multiply directly:
$c^{\lambda'}_{\mu',\nu'}(m)$ counts the number of ways to represent the vector
$(\lambda'_i)_{i=1,2,\dots}$ as a sum of $(\mu'_i)_{i=1,2,\dots}$ and a permutation
of the coordinates of the vector $(\nu'_i)_{i=1,2,\dots}$.

It remains to check that
\begin{multline}
\label{eq_x15}
  \sum_{\nu,\mu}
   c^{\lambda'}_{\mu',\nu'}(m) \prod_{i=1}^N (N+1-i)^{\lambda_i-\mu_i-\nu_i}
   \prod_{i=1}^N \left[ e_i^{\mu_{i}-\mu_{i+1}}(\a) e_i^{\nu_i-\nu_{i+1}}(\b) \right]
   \\
   = \prod_{\ell=1}^N \left[\sum_{p=0}^{\ell}
 e_p(\a) e_{\ell-p}(\b)
  \frac{N(N-1)\cdots(N-\ell+1)}{N(N-1)\cdots(N-p+1) \cdot N(N-1)\cdots(N-\ell+p+1)}
  \right]^{\lambda_\ell-\lambda_{\ell+1}},
\end{multline}
which is straightforward given the above description of
$c^{\lambda'}_{\mu',\nu'}(m)$.

\subsection{Proof of Theorem \ref{Theorem_crystallization}}
\label{Section_Proof_of_Th_crystal_2}

Let us start by computing the constant $Z_N$ in Definition \ref{def_betacorner}. For
that we use the Dixon--Anderson integration formula, see \cite{Dixon},
\cite[Exercise 4.2, q.\ 2]{For}, which reads
\begin{multline}
\label{eq_Dixon_formula}
 \int_T \prod_{1\le i<j\le n} (t_i-t_j) \prod_{i=1}^n \prod_{j=1}^{n+1}
 \frac{|t_i-a_j|^{\alpha_j-1}}{|b-t_i|^{\alpha_j}}
 dt_1\cdots dt_n
 \\
 = \dfrac{\prod\limits_{j=1}^{n+1}\Gamma(\alpha_j)}{\Gamma\left(\sum\limits_{j=1}^{n+1} \alpha_j\right)} \prod_{1\le i<j \le n+1}
 (a_i-a_j)^{\alpha_i+\alpha_j-1} \prod_{i=1}^{n+1} |b-a_i|^{\alpha_i-\sum_{j=1}^{n+1} \alpha_j} ,
\end{multline}
where the domain of integration $T$ is given by
$$
 a_1< t_1 <a_2<t_2\dots<t_n<a_{n+1}.
$$
Choosing $b=\infty$, $\alpha_i=\beta/2$, we verify \eqref{eq_normalization} by
induction in $N$.

\bigskip

The next step is to find $\tilde x_i^k$. We do this sequentially: first, for $k=N-1$
and $1\le i \le N-1$, then for $k=N-2$, etc, until we reach $k=1$. Using the
definition of the $\beta$--corners process and the formula for $Z_N$ we write down
the conditional distribution of $x_i^{k-1}$, $1\le i\le k-1$ given $x_i^k$, $1\le
i\le k$:
\begin{multline}
\label{eq_cond_denstity}
 P(x_1^{k-1},\dots,x_{k-1}^{k-1}\mid x_1^k,\dots,x_k^k)=
 \frac{\Gamma(k\beta/2)} {\Gamma(\beta/2)^{k}}
 \frac{\prod\limits_{1\le i <j \le k-1}(x_j^{k-1}-x_i^{k-1})}{\prod\limits_{1\le i<j \le k}
 (x_j^k-x_i^k)^{\beta-1}} \\ \times
 \left( \prod_{i=1}^k \prod_{j=1}^{k-1}\left|x_i^k-x_j^{k-1}\right|
 \right)^{\beta/2-1}.
\end{multline}
As $\beta\to\infty$, the density (as a function of $x_1^{k-1},\dots,x_{k-1}^{k-1}$)
concentrates near the point where the second line of \eqref{eq_cond_denstity} is
maximized, so we need to solve the maximization problem:
\begin{equation}
\label{eq_maximization}
 \prod_{i=1}^k \prod_{j=1}^{k-1}\left|x_i^k-x_j^{k-1}\right| \to \max, \quad \text{ with
 fixed } x_1^k,\dots,x_k^k.
\end{equation}
Taking logarithmic derivatives of \eqref{eq_maximization} in
$x_1^{k-1},\dots,x_{k-1}^{k-1}$ we find that the optimal point $(\tilde
x_1^{k-1},\dots,\tilde x_{k-1}^{k-1})$ should satisfy the $k-1$ equations
\begin{equation}
\label{eq_optim_equations}
 \sum_{j=1}^{k} \frac{1}{\tilde x_{i}^{k-1}-x_{j}^{k}} = 0,\quad i=1,\dots,k-1.
\end{equation}
Observe that \eqref{eq_optim_equations} is precisely the collection of equations that define the $k-1$ roots
of the derivative of the polynomial $f(u)=\prod_{j=1}^k (u-x^k_j)$. This implies the
first part of \eqref{eq_crystallization_formula} of Theorem
\ref{Theorem_crystallization}.

\medskip

For the second part we Taylor-expand the density of the $\beta$--corners process
near the point $\tilde x_i^k$, $1\le i\le k \le N$. Set
$$
 x_i^k=\tilde x_i^k+ \frac{ \Delta x_i^k}{\sqrt{\beta}}
$$
and rewrite \eqref{eq_beta_corners_def} as
\begin{equation} \label{eq_beta_inf_expansion}
\begin{split}
  \frac{1}{Z_N}
  \prod_{k=1}^{N-1} &\left[\prod_{1\le i<j\le k} (\tilde x_j^k-\tilde x_i^k)^{2-\beta}\right] \cdot \left[\prod_{a=1}^k \prod_{b=1}^{k+1}
 |\tilde x^k_a-\tilde x^{k+1}_b|^{\beta/2-1}\right]
 \\ \times \prod_{k=1}^{N-1}&\left[\prod_{1\le i<j\le k} \left(1+\frac{1}{\sqrt{\beta}}\cdot
 \frac{\Delta x_j^k-\Delta x_i^k}{\tilde x_j^k-\tilde x_i^k}\right)^{2-\beta}\right] \cdot \left[\prod_{a=1}^k \prod_{b=1}^{k+1}
 \left(1+\frac{1}{\sqrt{\beta}}\cdot \frac{\Delta x^k_a-\Delta x^{k+1}_b}{\tilde x^k_a-\tilde x^{k+1}_b}\right)^{\beta/2-1}\right]
 \\
 =&  \frac{1}{Z_N}
  \prod_{k=1}^{N-1} \left[\prod_{1\le i<j\le k} (\tilde x_j^k-\tilde x_i^k)^{2-\beta}\right] \cdot \left[\prod_{a=1}^k \prod_{b=1}^{k+1}
 |\tilde x^k_a-\tilde x^{k+1}_b|^{\beta/2-1}\right]
 \\ \times& \prod_{k=1}^{N-1} \exp\left[ -\sqrt{\beta}\sum_{1\le i<j\le k}
 \frac{\Delta x_j^k-\Delta x_i^k}{\tilde x_j^k-\tilde x_i^k} +\frac{\sqrt{\beta}}{2} \sum_{a=1}^k\sum_{b=1}^{k+1}
 \frac{\Delta x^k_a-\Delta x^{k+1}_b}{\tilde x^k_a-\tilde x^{k+1}_b}\right]
\\ \times& \prod_{k=1}^{N-1} \exp\left[ \sum_{1\le i<j\le k}
 \frac{(\Delta x_j^k-\Delta x_i^k)^2}{2(\tilde x_j^k-\tilde x_i^k)^2} - \sum_{a=1}^k\sum_{b=1}^{k+1}
 \frac{(\Delta x^k_a-\Delta x^{k+1}_b)^2}{4(\tilde x^k_a-\tilde x^{k+1}_b)^2}+O(\beta^{-1/2})\right]
\end{split}
\end{equation}
As $\beta\to\infty$, the last line of \eqref{eq_beta_inf_expansion} gives the
desired density of $\xi^k_i$ in the $\infty$--corners process, and it remains to
show that the fourth line of \eqref{eq_beta_inf_expansion} is identically equal to
$1$. Indeed, the coefficient of $\Delta x_i^k$ in the expression under exponent is
\begin{equation}
\label{eq_linear_term}
 -\sqrt{\beta}\sum_{\begin{smallmatrix}1\le j \le k,\\ j\ne i\end{smallmatrix}} \frac{1}{\tilde x_i^k-\tilde x_j^k}
 +\frac{\sqrt{\beta}}{2} \sum_{j=1}^{k-1}
 \frac{1}{\tilde x_i^k-\tilde x_j^{k-1}}+\frac{\sqrt{\beta}}{2} \sum_{j=1}^{k+1}
 \frac{1}{\tilde x_i^k-\tilde x_j^{k+1}}
\end{equation}
The last term in \eqref{eq_linear_term} is zero because of the equations
\eqref{eq_optim_equations}. For the first two terms, recall that $\tilde x_i^{k}$
and $\tilde x^i_{k-1}$ are roots of a polynomial and its derivative, i.e.\
\begin{equation}
\label{eq_x1}
 \sum_{a=1}^{k} \prod_{j\ne a} (u-\tilde x_j^{k}) = k \prod_{a=1}^{k-1} (u-\tilde x_a^{k-1})
\end{equation}
We then differentiate \eqref{eq_x1} in $u$ and plug in $u=\tilde x_i^{k}$ to get
\begin{equation}
\label{eq_x2}
 2 \sum_{a=1}^{k} \prod_{j \ne a, i} (\tilde x_i^{k}-\tilde x_j^{k})= k \sum_{j=1}^{k-1} \prod_{a\ne
 j} (\tilde x_i^{k}-\tilde x_a^{k-1})
\end{equation}
and plug $u=\tilde x_i^{k}$ into \eqref{eq_x1} to get
\begin{equation}
\label{eq_x3}
 \prod_{j\ne i} (\tilde x_i^{k}-\tilde x_j^{k}) = k \prod_{a=1}^{k-1} (\tilde x_i^{k}-\tilde
 x_a^{k-1}).
\end{equation}
Dividing \eqref{eq_x2} by \eqref{eq_x3}, we conclude that the first two terms in
\eqref{eq_linear_term} cancel out. As a result, we see that the contribution of the fourth
line of \eqref{eq_beta_inf_expansion} must vanish as $\beta\to\infty$.

\smallskip

The last ingredient of the proof of Theorem \ref{Theorem_crystallization} is to
explain that the density \eqref{eq_Gaussian_density} is integrable, i.e., that the
inverse covariance matrix arising in this density is indeed positive definite. This
would have been immediate, if all the terms in the exponent had negative signs, yet
the $i<j$ sum is positive, and therefore, an additional clarification is necessary.
We prove that the integral of \eqref{eq_Gaussian_density} is finite by integrating
the variables $\xi_{i}^k$ sequentially in $k$; that is, we first integrate over
$\xi_1^1$, then over $\xi^2_1$, $\xi^2_2$, etc.
Each step is an integration over
$k-1$ variables $\xi^{k-1}_1,\dots,\xi^{k-1}_{k-1}$, and the integral is
the limit (as $\beta\to\infty$) of the identity expressing the unit total mass of the
conditional probability
 \eqref{eq_cond_denstity}. Repeating the argument \eqref{eq_beta_inf_expansion}, this limit is an
 identity holding for any $k$ reals $\zeta_1,\dots,\zeta_k$:
 \begin{equation}
 \label{eq_Gaussian_evaluation}
  \int\dots\int \exp\left(-\sum_{a=1}^k \sum_{b=1}^{k-1}
  \frac{(\zeta_a-\xi_b)^2}{(x_a^k-x_b^{k-1})^2} \right) d\xi_1 \dots d\xi_{k-1}=Z\cdot\exp\left(
  -2\sum_{1\le a<b \le k} \frac{(\zeta_a-\zeta_b)^2}{(x_a^k-x_b^k)^2}\right),
 \end{equation}
 where $Z>0$ does not depend on $\zeta_1,\dots,\zeta_k$; it can be explicitly evaluated, but we do not need its value here.
  Note that in the integrand of \eqref{eq_Gaussian_evaluation} the expression in the exponent is clearly negative, and therefore
 the question of convergence of the integral does not arise. However, iteratively
 using \eqref{eq_Gaussian_evaluation} for $k=2,\dots,N$, we compute the (finite)
 normalizing constant for the density \eqref{eq_Gaussian_density}.

\section{Discrete versions and generalities}
\label{Section_Discrete}

\subsection{Expectation identities at general $(q,t)$}

The main ingredient of our proofs, which is the expectation computations of Section
\ref{Section_expectations}, admits a generalization up to the hierarchy of symmetric
functions to the level of Macdonald polynomials.

For Propositions \ref{Proposition_Jack_observable_product},
\ref{Proposition_Jack_observable_sum} the Macdonald version is as follows.

\begin{proposition} \label{Proposition_Mac_observable_product} Fix signatures
$\lambda=(\lambda_1\ge\dots\ge\lambda_N)$ and $\mu=(\mu_1\ge \dots\ge\lambda_N)$,
and let $\nu=(\nu_1\ge\dots\nu_N)$ be a random signature with distribution given by
the weight $\nu\mapsto c^{\nu}_{\lambda,\mu}(\hat P;q,t)$.
 For each positive signature $\rho=\rho_1\ge\rho_2\ge\dots\ge\rho_N\ge 0$ we have
\begin{multline} \label{eq_Mac_observable_product}
 \E P_\rho(q^{\nu_1} t^{N-1},q^{\nu_2} t^{N-2},\dots,q^{\nu_N};\, q,t)\\=
 \frac{P_\rho(q^{\lambda_1} t^{N-1},q^{\lambda_2} t^{N-2},\dots,q^{\lambda_N};\, q,t) \cdot P_\rho(q^{\mu_1} t^{N-1},q^{\mu_2} t^{N-2},\dots,q^{\mu_N};\, q,t)}
 {P_\rho(t^{N-1},t^{N-2},\dots,1;\, q,t)}.
\end{multline}
\end{proposition}
\begin{proof}
 Using \eqref{eq_Macd_symmetry} and definition of $c^{\nu}_{\lambda,\mu}(\hat P;q,t)$ we have
\begin{multline*}
 \E \left[ \frac{P_\rho(q^{\nu_1} t^{N-1},\dots,q^{\nu_N};\, q,t)}
 {P_\rho(t^{N-1},\dots,1;\, q,t)}\right]=
  \E \left[\frac{P_\nu(q^{\rho_1} t^{N-1},\dots,q^{\rho_N};\, q,t)}
 {P_\nu(t^{N-1},\dots,1;\, q,t)}\right]\\=\sum_{\nu} \frac{P_\nu(q^{\rho_1} t^{N-1},\dots,q^{\rho_N};\, q,t)}
 {P_\nu(t^{N-1},\dots,1;\, q,t)} c^{\nu}_{\lambda,\mu}(\hat P;q,t)
 = \frac{P_\lambda(q^{\rho_1} t^{N-1},\dots,q^{\rho_N};\, q,t)}
 {P_\lambda(t^{N-1},\dots,1;\, q,t)} \cdot \frac{P_\mu(q^{\rho_1} t^{N-1},\dots,q^{\rho_N};\, q,t)}
 {P_\mu(t^{N-1},\dots,1;\, q,t)}\\= \frac{P_\rho(q^{\lambda_1} t^{N-1},\dots,q^{\lambda_N};\, q,t)}
 {P_\rho(t^{N-1},\dots,1;\, q,t)} \cdot \frac{P_\rho(q^{\lambda_1} t^{N-1},\dots,q^{\lambda_N};\, q,t)}
 {P_\rho(t^{N-1},\dots,1;\, q,t)}. \qedhere
\end{multline*}
\end{proof}

For an analogue of Proposition \ref{Proposition_Jack_observable_projection} we need
a new definition generalizing $\pi^{\beta}_{N\to k}$.

\begin{definition}
 Fix an $0<k<N$, and a signature $\lambda=(\lambda_1\ge y_2\ge\dots\ge \lambda_N)$.
 Define a random signature $\nu=(\nu_1\ge\nu_k)$ with distribution $\pi^{q,t}_{N\to
 k}(\lambda)$ through the following decomposition
 \begin{equation}
 \label{eq_Macd_projection_def}
  \frac{P_\lambda(t^{1-N},t^{2-N},\dots,t^{-k},z_1,\dots,z_k;\, q,t)}{P_\lambda(t^{1-N},\dots,t^{-1},1;\,q,t)}=
  \sum_{\nu} \pi^{q,t}_{N\to
 k}(\lambda) [\nu]  \cdot
  \frac{P_\nu(z_1,\dots,z_k;\, q,t)}{P_\nu(t^{1-k},\dots,1;\, q,t)}.
 \end{equation}
\end{definition}

Plugging $z_i=t^{1-i}$ into \eqref{eq_Macd_projection_def} one proves that $\sum_\nu
\pi^{q,t}_{N\to
 k}(\lambda) [\nu]=1$. The non-negativity of weights follows from the branching rules for the Macdonald
polynomials, see \cite{M}.

Let us emphasize, that we use negative powers of $t$ in
\eqref{eq_Macd_projection_def}. On the other hand, the normalization of Macdonald
polynomials $\hat P_\lambda$ entering the definition of $c^{\nu}_{\lambda,\mu}(\hat
P;\, q,t)$ involved positive powers. For the latter this difference is not important
due to homogeneity of $P_\lambda$ and vanishing of $c^{\nu}_{\lambda,\mu}(\hat P;\,
q,t)$ unless $\sum_{i} (\lambda_i+\mu_i-\nu_i)=0$. However, for
\eqref{eq_Macd_projection_def} this becomes important.

\begin{proposition}
\label{Proposition_Mac_observable_projection}
 Fix $k<N$, a signature $\lambda$ of rank $N$, and let $\nu$ be $\pi^{q,t}_{N\to
 k}(\lambda)$--distributed. Then for any non-negative signature $\rho=(\rho_1\ge\dots\ge\rho_k\ge 0)$ of rank $k$ we have
 \begin{equation}
  \E P_{\rho}(q^{\nu_1}t^{k-1},\dots, q^{\nu_k};\, q,t)= \frac{P_\rho(t^{k-1},\dots,t,1;\,
  q,t)}{{P_\rho(t^{N-1},\dots,t,1;\,
  q,t)}}  \cdot P_\rho(q^{\lambda_1}t^{N-1},\dots,q^{\lambda_N};\,
  q,t)
 \end{equation}
 where we also treated $\rho$ as a signature of rank $N$ by adding $N-k$ zero
 coordinates.
\end{proposition}
\begin{proof}
 Using \eqref{eq_Macd_symmetry} and homogeneity of Macdonald polynomials we have
\begin{multline}
  \E\left[\frac{P_{\rho}(q^{\nu_1}t^{k-1},\dots, q^{\nu_k};\, q,t)}
  {P_{\rho}(t^{k-1},\dots,t, 1;\, q,t)}\right]=
    \E\left[\frac{P_{\nu}(q^{\rho_1}t^{k-1},\dots, q^{\rho_k};\, q,t)}
  {P_{\nu}(t^{k-1},\dots, 1;\, q,t)}\right]\\
  =\sum_{\nu} \pi^{q,t}_{N\to
 k}(\lambda) [\nu] \frac{P_{\nu}(q^{\rho_1}t^{k-1},\dots, q^{\rho_k};\, q,t)}
  {P_{\nu}(t^{k-1},\dots,t, 1;\, q,t)}=\sum_{\nu} \pi^{q,t}_{N\to
 k}(\lambda) [\nu] \frac{P_{\nu}(q^{\rho_1},\dots, q^{\rho_k} t^{1-k};\, q,t)}
  {P_{\nu}(t^{1-k},\dots,t^{-1}, 1;\, q,t)}\\
 =\frac{P_{\lambda}(q^{\rho_1},\dots, q^{\rho_k} t^{1-k}, t^{-k},\dots, t^{1-N};\, q,t)}
  {P_{\lambda}(t^{1-N},\dots,t^{-1}, 1;\, q,t)}
 =\frac{P_{\lambda}(q^{\rho_1} t^{N-1},\dots, q^{\rho_k} t^{N-k}, t^{N-k-1},\dots, 1;\, q,t)}
  {P_{\lambda}(t^{N-1},\dots,t, 1;\, q,t)}
 \\= \frac{P_{\rho}(q^{\lambda_1} t^{N-1},\dots, q^{\lambda_N};\, q,t)}
  {P_{\rho}(t^{N-1},\dots,t, 1;\, q,t)}. \qedhere
\end{multline}
\end{proof}
An analogue of Proposition \ref{Proposition_Mac_observable_projection} at $q=t$ is
implicitly used in \cite{GO}, \cite{Olsh_qGT}, \cite{Olsh_qMac} for the study of the
extended Gelfand--Tsetlin graph.

\smallskip

As in Section \ref{Section_Proof_of_Th_independence}, if we choose
$\rho_1=\dots=\rho_\ell=1$, $\rho_{\ell+1}=\rho_{\ell+2}=\dots=0$ in Propositions
\ref{Proposition_Mac_observable_product},
\ref{Proposition_Mac_observable_projection}, then the Macdonald polynomials would
turn into elementary symmetric functions $e_\ell$, and we get formulas for the
expectations of $e_\ell$. In particular, $q$ does not enter into these formulas in
any explicit form, which is a $(q,t)$--analogue of the $\beta$--independence in
Theorem \ref{Theorem_independence}.

\subsection{Crystallization for general $(q,t)$} The Law of Large
Numbers (crystallization) of Theorems \ref{Theorem_crystal},
\ref{Theorem_crystallization} is obtained from operations on Macdonald polynomials
$P_\lambda(\cdot;\, q,t)$ by a triple limit transition:
\begin{equation}
\label{eq_three_limits}
 q\to 1; \quad \quad t=q^{\theta},\, \theta\to +\infty;\quad \quad
 \lambda_i=\eps^{-1} r_i,\, \eps\to 0.
\end{equation}
In these theorems, we made the limit transitions in a particular order (first, $q\to
1$, $\lambda_i\to \infty$ to degenerate into random matrices, and only then
$\theta\to\infty$), but different orders of taking limits are also possible and
would lead to \emph{another} set of answers. We do not address the full
classification of the limiting behaviors here (it probably deserves a separate
publication), but only mention two possible scenarios.

\begin{enumerate}
\item If we start with $\theta\to\infty$ (so $t\to 0$), then Macdonald polynomials degenerate to $q$--Whittaker
functions, as discussed in details in \cite{GLO}, \cite{BigMac}. Two different
further $q\to 1$ limits were studied in the literature. The first one is parallel to
the degeneration of $q$--Whittaker functions to Whittaker functions: the particles
crystallize on a perfect lattice, while fluctuations are related to directed
polymers in random media, see \cite{BigMac}. Another limit in \cite{BorCorFer} leads
to more complicated Law of Large Numbers and Gaussian fluctuations.

\item We can first degenerate Macdonald polynomials into Jacks and the latter into
products of elementary symmetric functions, as in \eqref{eq_Jack_infty}. After
taking these limits, an analogue of the $\beta$--corners process would involve
weights given by products of Binomial coefficients, while the top--row (which was
$y_1<\dots<y_N$ in Definition \ref{def_betacorner}) is still discrete. Linearly
rescaling the coordinates of the top row one finds yet another Law of Large Numbers.
Using Stirling's formula and solving the associated maximization formula (as in the
proof of Theorem \ref{Theorem_crystallization}) one can explicitly find the limit
then.
It has the following description: $k$th particle of level $M-1$ splits the interval
between $k$th and $(k+1)$st particles on level $M$ in the proportion $k:(N-k)$.


\end{enumerate}

\end{document}